\documentclass[11pt]{article}
\usepackage{amssymb}
\usepackage{amsmath}
\usepackage{color}
\usepackage{hyperref}
\usepackage{csquotes}
\usepackage{IEEEtrantools}

\newenvironment{proof}{\noindent {\bf Proof }}
{\hfill $\bullet$ \vspace{0.25cm}}

\newcommand{\dis}{\displaystyle}

\newtheorem{thm}{Theorem}
\newtheorem{prop}{\indent Proposition}

\newtheorem{rem}{\indent Remark}
\newtheorem{lem}{\indent Lemma}

\newcommand{\mmmintone}[1]{{\dis{\int\kern -.36cm-}}_{\kern-.21cm\substack{#1}}\;\;}
\newcommand{\mmmintwo}[2]{{\dis{\int\kern -.43cm-}}_{\kern-.21cm\substack{#1}}^{\substack{#2}}\;\;}
\newcommand{\submint}{{\scriptstyle{\int\kern -.66em -}}}
\newcommand{\submintone}[1]{{\scriptstyle{\int\kern -.66em-}}_{\scriptscriptstyle{\kern-.21em\substack{#1}}}}
\newcommand{\fracmint}{{\textstyle{\int\kern -.88em -}}}
\newcommand{\fracmintone}[1]{{\textstyle{\int\kern -.88em
-}}_{\scriptscriptstyle{\kern-.21em\substack{#1}}}\;}

\title{First Passage Time Densities through H\"{o}lder curves}

\author{Jimyeong Lee\footnote{  E-mail: jimyeong.lee@gssi.infn.it }\\ \\Gran Sasso Science Institute, Viale Francesco Crispi, 7, 67100, L'Aquila, Italy}

\date{\today}

\begin{document}

\maketitle

\begin{abstract}
We prove that for a standard Brownian motion, there exists a first-passage-time density function through a H\"{o}lder curve with exponent greater than $1/2$. By using a property of local time of a standard Brownian motion and the theories of partial differential equations in Cannon $\cite{Cannon}$, we find a sufficient condition for existence of the density function. We also show that this density function is proportional to the space derivative of the Green function of the heat equation with Dirichlet boundary condition at the moving boundary. 
\end{abstract}

\vskip 1cm

\section{Introduction}
\label{Fsec.1}
In this paper we will study the probability of the hitting time to a moving boundary. To state the main result, we need some notations on a Brownian motion. Let us call $P_{r,s}, r\in\mathbb{R}, s\geq 0$, the law on $\displaystyle C([s,\infty))$ of the Brownian motion $B_t$, $t\geq s$, which starts from $r$ at time $s$, i.e. $B_s=r$. For each $t>s$ the law of $B_t$ is absolutely continuous with respect to the Lebesgue measure and has a density $G_{s,t}(r,\cdot)$ which is the Gaussian $\displaystyle G(\cdot,t;r,s)=\frac{1}{\sqrt{2\pi (t-s)}}\exp{\left(-\frac{(\cdot-r)^{2}}{2(t-s)}\right)}$.  We denote by $E_{r,s}$ the expectation under $P_{r,s}$. For a given curve $\displaystyle X=\{t\rightarrow X_t\}$, $s\geq 0$ and $r<X_s$, we define
\begin{eqnarray}
\displaystyle \tau_{r,s}^{X}=\inf\{t \geq s: B_t\geq X_t \},\ and=\infty\ \textrm{if the set is empty},
\end{eqnarray}
where $B_s=r$ and denote by $F_{r,s}^{X}(dt)$ the distribution of $\displaystyle \tau_{r,s}^{X}$ induced by $P_{r,s}$. For s=0, we use abbreviated forms  $P_{r}$, $E_{r}$, $\displaystyle \tau_{r}^{X}$, $F_{r}^{X}(dx)$  instead of $P_{r,0}$, $E_{r,0}$, $\displaystyle \tau_{r,0}^{X}$, $F_{r,0}^{X}(dx)$ respectively whenever it is needed. In addition, for $t>0$, let us call $d\mu_{r_0}(\cdot,t)$ the positive measure on $(-\infty,X_t)$ such that 
\begin{eqnarray}
\label{1.0}
\displaystyle \int_{(-\infty,X_t)} d\mu_{r_0}(x,t)f(x)=E_{r_0}[f(B_t);\tau_{r_0}^{X}> t]
\end{eqnarray}
for all $\displaystyle f\in C_c^{\infty}(\mathbb{R})$ with $\displaystyle\textrm{supp}\hspace{0.5mm}f\Subset (-\infty,X_t)$. The main result in the paper is;
\begin{thm} 
\label{mainthm}
If $X$ is H\"{o}lder continuous on any finite interval in $[0,\infty)$ with exponent $\displaystyle\gamma\in\left(1/2,1\right]$ and $r_0<X_0$, then
\begin{enumerate}
\item\label{1.1} $d\mu_{r_0}(x,t)=G_{0,t}^{X}(r_0,x)dx$ where for all $x<X_t$,
\begin{eqnarray}
\label{12345}
\displaystyle G_{0,t}^{X}(r_0,x)=G_{0,t}(r_0,x)-\int_{[0,t)}F_{r_0}^{X}(ds)G_{s,t}(X_{s},x).
\end{eqnarray}
\item\label{1.2} $\displaystyle F_{r_0}^{X}(ds)$ has a density function $p$ on $[0,\infty)$, namely $F_{r_0}^{X}(ds)=p(s)ds$.\\
\item\label{1.3} $\displaystyle p(t)=-\frac{1}{2}\frac{\partial}{\partial x}G_{0,t}^{X}(r_0,x)\Big|_{x=X_t^{-}}$ for all\ $t>0$.\\
\item\label{1.4} $\displaystyle G_{0,t}^{X}(r_0,x)$ solves 
\begin{eqnarray} 
\label{1}\displaystyle v_t=\frac{1}{2}v_{xx},\ \ -\infty<x<X_t,\ t>0, \\
\label{2}\displaystyle \lim_{(x,t)\rightarrow (X_s,s)}v(x,t)=0,\  s>0,\\
\label{3} \displaystyle \lim_{(x,t)\rightarrow (y,0)}v(x,t)=\delta_{r_0}(y),\ -\infty<y<X_0.
\end{eqnarray}
\end{enumerate}
\end{thm}
\textbf{Remarks}
\vskip 0.1cm
Item $\ref{1.4}$ of $\textbf{Theorem~\ref{mainthm}}$ states that for any $r_0<X_0$ the function $\displaystyle G_{0,\cdot}^{X}(r_0,\cdot)$ given by $(\ref{12345})$ is the Green function of the heat equation with Dirichlet boundary conditions at the moving boundary $X$. In $\cite{AP}$, when a moving boundary is infinitely differentiable, it is showed that the space derivative of the Green function of the heat equation at the boundary is proportional to the hitting time density function. Likewise by items  $\ref{1.2}$ and  $\ref{1.3}$ of $\textbf{Theorem~\ref{mainthm}}$, the space derivative of $\displaystyle G_{0,\cdot}^{X}(r_0,\cdot)$ at the moving boundary is proportional to the hitting time density function $p$.\\
For the case when $X_t=a+bt$, it is well known(see for instance $\cite{KS}$) that for $r<a$, $\displaystyle\tau_r^{X}$ has a probability density function given by $\displaystyle f(t)=\frac{a-r}{\sqrt{2\pi t^3}}\exp{\left(-\frac{(a+bt-r)^2}{2t}\right)}\mathbf{1}_{t>0}$. In addition, there is another result in $\cite{RSS}$ when $X_t=a+bt^{\frac{1}{p}}$ for $p\geq 2$, $r<a$, then $\displaystyle\tau_r^{X}$ has a probability density function. In \cite{BD}, it is showed that if the boundary behaves as a Lipschitz curve in a local time, then the first passage time density can be expressed explicitly.
\\
In $\cite{PS}$, it is proved that for any continuous curve $X_t$ and $r<X_0$, there is a distribution $\displaystyle F_r^{X}$  of $\displaystyle \tau_{r}^{X}$ which satisfies the following integral equation(called the Master Equation):
\begin{eqnarray}
\label{0000}
\displaystyle\Psi\left(\frac{z-r}{\sqrt{t}}\right)=\int_{0}^{t}\Psi\left(\frac{z-X_s}{\sqrt{t-s}}\right)F_r^{X}(ds),
\end{eqnarray}
where $z\geq X_t$, $t>0$ and $\displaystyle\Psi(z)=\int_z^{\infty}\frac{1}{\sqrt{2\pi}}\exp{\left(-\frac{x^{2}}{2}\right)}dx$. This can be proved intuitively as follows; the left hand side of $(\ref{0000})$ is the probability that a Brownian motion starts at $r$ at time 0 and reaches $z$ greater or equal to $X_t$ at time $t$. Then it should hit the boundary at least once which implies the right hand side of of $(\ref{0000})$.
Moreover, it is showed that if $X_t$ is $\displaystyle C^{1}$, then there exists a continuous density function $f$ of  $\displaystyle F_r^{X}$. There is an extension of this result, in $\cite{RSS}$, to curves $X$ which are differentiable with $\displaystyle \left|\frac{dX_t}{dt}\right|\leq Ct^{-\alpha}$ for some constant $C>0$ and $\alpha<1/2$.\\
By $(\ref{0000})$, for any continuous curve $X_t$ and $r<X_0$, we have
\begin{eqnarray}
\label{2.3}
\displaystyle\Psi\left(\frac{X_t-r}{\sqrt{t}}\right)=\int_{0}^{t}\Psi\left(\frac{X_t-X_s}{\sqrt{t-s}}\right)F_r^{X}(ds).
\end{eqnarray}
which can be regarded as an integral equation for $\displaystyle F_r^{X}(ds)$.
In $\cite{TM}$, the equation $(\ref{2.3})$ is studied when $X$ is H\"{o}lder continuous with exponent greater than $1/2$. It is proved that there exists a unique continuous function $q$ such that
\begin{eqnarray}
\displaystyle \Psi\left(\frac{X_t-r}{\sqrt{t}}\right)=\int_{0}^{t}\Psi\left(\frac{X_t-X_s}{\sqrt{t-s}}\right)q(s)ds.
\end{eqnarray}
To conclude that $F_r^{X}(ds)=q(s)ds$, one still needs that $\displaystyle F_r^{X}(ds)$ is absolutely continuous with respect to Lebesgue measure. The analysis in $\cite{TM}$ as well as the proof of $\textbf{Theorem~\ref{mainthm}}$ uses extensively the work by Cannon $\cite{Cannon}$ on the heat equation with the moving boundary.

\vskip1cm

\setcounter{equation}{0}

\section{Preliminaries: a weaker form of \textbf{Theorem~\ref{mainthm}}} 

\label{Fsec.2}

We define $D$ as
$$\displaystyle D:=\{(x,t):-\infty<x<X_t,\ t>0\}$$
and
\begin{eqnarray}
\label{2345}
\displaystyle u(x,t):=\int_{-\infty}^{0}h(\xi)G_{0,t}^{X}(\xi,x)d\xi
\end{eqnarray}
for given $\displaystyle h\in C_c^{\infty}((-\infty,X_0);\mathbb{R}_{+})$ and all $(x,t)\in D$.\\
If $X$ is H\"{o}lder continuous with exponent $\gamma\in[1/2,1]$, then $u$ solves the heat equation in $D$ with the initial condition $h$ and Dirichlet boundary condition.(See $\textbf{Theorem~\ref{thm1}}$ below) However uniqueness fails for the heat equation with the initial function $h$ and Dirichlet boundary conditions (see $\textbf{Remark~\ref{rem1} and \ref{rem2}} $ below).\\
To have the uniqueness, first of all, we restrict time variable of $D$ in a finite interval. Thus we fix $T>0$ and define the parabolic cylinder $D_T$ which is a subset of $D$ as
$$\displaystyle D_T:=\{(x,t):-\infty<x<X_t,\ 0<t\leq T\}.$$
Consider the following initial-boundary value problem 
\begin{equation}
\left\{ \,
\begin{IEEEeqnarraybox}[][c]{l?s}
\IEEEstrut
\displaystyle v\in C(\overline{D_T})\cap C^{2,1}(D_T), \\
\displaystyle v_t=\frac{1}{2}v_{xx},\ \ (x,t)\in D_T, \\
 \displaystyle v(X_t,t)=0,\ 0 < t \leq T,\\ 
  \displaystyle v(x,0)=h(x),\ -\infty<x<X_0,\\ 
  \displaystyle  \lim_{x\rightarrow -\infty}\sup_{0< t< T}|v(x,t)|=0.
\IEEEstrut
\end{IEEEeqnarraybox}
\right.
\label{4242}
\end{equation}

We prove the following weaker form of $\textbf{Theorem~\ref{mainthm}}$.
\begin{thm}
\label{thm1}
Let $X$ be a H\"{o}lder continuous curve on any finite interval in $[0,\infty)$ with exponent $\gamma\in[1/2,1]$ and let $X_0=0$. 
\begin{enumerate}
\item\label{thm1.1}The function $u$ defined in $(\ref{2345})$ is the unique solution of $(\ref{4242})$.\\
\item\label{thm1.2}If $\gamma\in(1/2,1]$, then $u$ has the left hand derivative at the boundary $u_x({X_t^{-}},t)$ which is continuous on $(0,\infty)$. Moreover, for all $t > 0$, $\displaystyle p_{h}(t):=-\frac{1}{2}{u_x}(X_{t}^{-},t)$ satisfies
\begin{eqnarray}
\label{888}
\displaystyle p_{h}(t)=-\int_{-\infty}^{0}h(\xi)G_x(X_t,t;\xi,0)d\xi+\int_{0}^{t}G_x(X_t,t;X_\tau,\tau)p_{h}(\tau)d\tau.
\end{eqnarray}

\end{enumerate}

\end{thm}

\begin{rem}
\label{rem1}
The uniqueness for $(\ref{4242})$ is not guaranteed if we do not assume $\displaystyle v\in C(\overline{D_T})$. When $X_t=0$ for all $t$ and $h$ is identically $0$, if $v(x,t)$ is given by $\displaystyle \frac{1}{\sqrt{2\pi t}}\left\{-\frac{x}{t}\right\}\exp\left(-\frac{x^{2}}{2t}\right)$, then this satisfies the heat equation with the initial data $0$ and is also $0$ on the boundary, but this is not continuous at $(0,0)$. 
\end{rem}
\begin{rem}
\label{rem2}
We need the condition $\displaystyle \lim_{x\rightarrow -\infty}\sup_{0< t< T}|v(x,t)|=0$ to have uniqueness. Indeed, for $X_t=0$ for all $t\geq 0$, the function
$$\displaystyle v(x,t)=\sum_{n=0}^{\infty}f^{(n)}\left(\frac{t}{2}\right)\frac{x^{2n+1}}{(2n+1)!},$$
where
$$
\displaystyle f(t)=
\begin{cases}
\displaystyle \exp{\left(-\frac{1}{t^{2}}\right)},\  \ t>0, \\
0,\hspace{7.5mm}t\leq0.
\end{cases}
$$
satisfies the heat equation with the initial data $0$ and is also $0$ on the boundary. Furthermore, $\displaystyle v\in C(\overline{D_T})\cap C^{2,1}(D_T)$.
\end{rem}
\begin{rem}
\label{rem.3}
It can be shown $\displaystyle\lim_{t\rightarrow 0}p_h(t)=0$ by $(\ref{888})$ and $\textbf{Lemma}~\ref{lem1}$ below. Thus $u_x({X_t^{-}},t)$ is continuous on $[0,\infty)$.
\end{rem}

Following $\cite{Cannon}$, we introduce $\displaystyle C_{(\nu)}^{0}\left((0,T]\right)$, $0<\nu \leq 1$, as the subspace of $C((0,T])$ that consists of those functions $\varphi$ such that
$$\displaystyle \lVert \varphi\rVert_{T}^{(\nu)}=\sup_{0<t\leq T}t^{1-\nu}|\varphi(t)|<\infty.$$
Then $\displaystyle C_{(\nu)}^{0}\left((0,T]\right)$ is a Banach space under the norm $\displaystyle \lVert \cdot \rVert_{T}^{(\nu)}$. 

We also introduce the following lemmas from $\cite{Cannon}$ which play an essential role in our analysis.

\begin{lem}[\emph{jump relation}]
For $\displaystyle\varphi\in C_{(\nu)}^{0}\left((0,T]\right)$, we have
\begin{eqnarray}
\lim_{x\rightarrow X_t^{\pm}}\frac{\partial w_\varphi}{\partial x}(x,t)=\mp\varphi(t)+\int_{0}^{t}G_x(X_t,t;X_\tau,\tau)\varphi(\tau)d\tau,
\end{eqnarray}
where $\displaystyle w_\varphi(x,t)=\int_{0}^{t}G(x,t;X_\tau,\tau)\varphi(\tau)d\tau$.
\end{lem}
 For two continuous curves $s_1$, $s_2$ such that $s_1(t)<s_2(t), t\in[0,T]$, let us set $\displaystyle E_T:=\{(x,t): s_{1}(t)<x<s_{2}(t),\ 0<t\leq T\}$ and $\displaystyle B_T:=\{(s_{i}(t),t): 0\leq t\leq T,\ i\in\{1,2\}\}\cup\{(x,0): s_{1}(0)< x< s_{2}(0)\}$.

\begin{lem}[\emph{The Weak Maximum(Minimum) Principle}]
For a solution $u$ of  $\displaystyle u_t=\frac{1}{2}u_{xx}$ in $E_T$, which is continuous in $E_T\cup B_T$,
\begin{eqnarray}
\displaystyle \max_{E_T\cup B_T}u= \max_{B_T}u.\ \ \left(\min_{E_T\cup B_T}u= \min_{B_T}u.\right)
\end{eqnarray}
\end{lem}

Before going to the proof of $\textbf{Theorem~\ref{thm1}}$, we need the following proposition.

\begin{prop}
\label{3456}
Let $X$ be a H\"{o}lder continuous curve on any finite interval in $[0,\infty)$ with exponent $\gamma\in[1/2,1]$. If the starting point of the Brownian motion is close to $X$, the first hitting time converges to 0. Precisely, $\displaystyle\lim_{\xi\rightarrow X_{0}}P_{\xi}\left[\tau_{\xi}^{X}> s\right]=0$ for all $s>0$.
\end{prop}
\begin{proof}
Without loss of generality, we may reduce this problem as the case for Brownian motion starting at $0$ and $X_0=\epsilon>0$ and let $\epsilon$ go to $0$. For $s>0$, let $\displaystyle m:=\sup_{0\leq t_1<t_2\leq s}\frac{|X_{t_2}-X_{t_1}|}{|t_2-t_1|^{\gamma}}$. Since $\displaystyle\limsup_{t\downarrow 0}\frac{B_t}{\sqrt{t}}=\infty$ a.e., for $M>m$, we have a sequence ${t_k}\downarrow0$ such that $M\sqrt{t_k}\leq B_{t_k}$ a.e. and $M\sqrt{t_k}-mt_k^{\gamma}\downarrow 0$ for all $k$. Thus, for $0<t\leq s$, we deduce that a.e.
$$\displaystyle \sup_{0\leq l\leq t}{\left\{B_l-X_l\right\}}\geq \sup_{0\leq l\leq t}{\left\{B_l-(\epsilon+ml^{\gamma})\right\}} \geq\sup_{t_k\leq t}{\left\{B_{t_k}-(\epsilon+mt_k^{\gamma})\right\}}\geq\sup_{t_k\leq t}{\left\{M\sqrt{t_k}-(\epsilon+mt_k^{\gamma})\right\}}.$$
Therefore,
$$\displaystyle P_{0}\left(\sup_{0\leq l\leq t}{\left\{B_l-X_l\right\}}<0\right)\leq P_{0}\left(\sup_{t_k\leq t}{\left\{M\sqrt{t_k}-mt_k^{\gamma}\right\}<\epsilon}\right).$$
For each sufficently small $\epsilon>0$, there is a greatest $k(\epsilon)$ such that $\displaystyle M\sqrt{t_{k(\epsilon)}}-mt_{k(\epsilon)}^{\gamma}\geq\epsilon$. Thus we obtain that $\tau_0^{X}\leq t_{k(\epsilon)}$ a.e. for all sufficiently small $\epsilon>0$. Since $k(\epsilon)$ is an increasing function as $\epsilon$ decreases and  $\displaystyle \lim_{\epsilon\rightarrow 0}k(\epsilon)=\infty$, the proposition follows.
\end{proof}

\begin{proof}\hspace{-1mm}
\textbf{of Theorem~\ref{thm1}}\\
Let us define $X_{\tau}^{\prime}:=X_{t-\tau}$ for all $0\leq\tau\leq t$. Using the invariance of the law of the Brownian motion under time reversal, we have
\begin{eqnarray}
\displaystyle u(x,t)=\int_{-\infty}^{0} h(r^{\prime})G_{0,t}^{X}(r^{\prime},x)dr^{\prime}=E_{x}[h(B_t);\tau_x^{X^{\prime}}> t].
\end{eqnarray}
Using this equality, we also have
\begin{eqnarray}
\displaystyle\left|u(x,t)\right|=\left|E_{x}[h(B_t);\tau_x^{X^{\prime}}> t]\right|\leq \lVert h\rVert_{\infty}P_{x}[\tau_x^{X^{\prime}}> t].
\end{eqnarray} 
For $s>0$, let us choose $\displaystyle 0<s^{*}<s$. Then for all $(x,t)$ sufficiently close to $(X_s,s)$, we obtain
\begin{eqnarray}
P_{x}[\tau_x^{X^{\prime}}> t]\leq P_{x}[\tau_x^{X^{\prime}}> s^{*}]
\end{eqnarray} 
which vanishes when $\displaystyle (x,t)\rightarrow (X_{s},s)$ by $\textbf{Proposition~\ref{3456}}$ so that $\displaystyle\lim_{(x,t)\rightarrow(X_s,s)}u(x,t)=0$.\\
In addition, we have
\begin{eqnarray}
\displaystyle\left|u(x,t)\right|=\left|E_{x}[h(B_t);\tau_x^{X^{\prime}}> t]\right| \leq E_{x}[ | h(B_t)|]=\int_{-\infty}^{0} |h(\xi)|G_{0,t}(x,\xi)d\xi
\end{eqnarray} 
which also vanishes when $\displaystyle (x,t)\rightarrow (0,0)$, since the support of $h$ is strictly away from $0$.
To prove that $u$ satisfies the initial data $h$, we write $\displaystyle y_t=\min_{s\in[0,t]}X_s^{\prime}$. For $x<0$ and any positive $\lambda>0$,
\begin{eqnarray}
\label{3.13}
P_{x}\left[\tau_x^{X^{\prime}}\leq t\right]\leq P_{x}\left[\max_{s\in[0,t]}B_s\geq y_t\right]= P_{x}\left[\max_{s\in[0,t]}\exp{(\lambda B_s)}\geq \exp{(\lambda y_t)}\right].
\end{eqnarray} 
Since the exponential of Brownian motion is a positive submartingale, we can apply Doob's inequality, then
\begin{eqnarray}
\label{3.14}
P_{x}\left[\max_{s\in[0,t]}\exp{(\lambda B_s)}\geq \exp{(\lambda y_t)}\right]\leq \frac{E_{x}[\exp{(\lambda B_t)}]}{\exp{(\lambda y_t)}}=\exp{\left(\frac{1}{2}\lambda^{2}t-\lambda(y_t-x)  \right)}. 
\end{eqnarray} 
From $(\ref{3.13})$ and $(\ref{3.14})$, we get
\begin{eqnarray}
\displaystyle\lim_{(x,t)\rightarrow (y,0)}P_{x}\left[\tau_x^{X^{\prime}}\leq t\right] \leq \exp{(-\lambda(X_0-y))}= \exp{(\lambda y)}
\end{eqnarray} 
so that the left hand side vanishes since $\lambda > 0$ is arbitrary and we are assuming $y < X_0 = 0$.
Thus we obtain that
\begin{eqnarray}
\displaystyle \lim_{(x,t)\rightarrow (y,0)}\int_{-\infty}^{0}h(\xi)G_{0,t}^{X}(\xi,x)d\xi=\lim_{(x,t)\rightarrow (y,0)}E_{x}[h(B_t)]=h(y).
\end{eqnarray} 
By the properties of the Gaussian kernel, we deduce that $u$ solves $(\ref{4242})$. We now prove uniqueness. If $v_1$, $v_2$ satisfy all the above conditions, then $\displaystyle v_1-v_2\in C(\overline{D_T})\cap C^{2,1}(D_T)$ satisfies the heat equation with the initial data $0$ and is also $0$ on the boundary. Moreover, $\displaystyle \lim_{x\rightarrow -\infty}\sup_{0< t< T}|v_1(x,t)-v_2(x,t)|=0$ so that by the weak maximum(minimum) principle, $v_1-v_2$ is all $0$ in $D_T$. Therefore, $v_1=v_2$ so that item $\ref{thm1.1}$ is proved.\\
To prove item $\ref{thm1.2}$, assuming that $\gamma\in(1/2,1]$, we define, $(x,t)\in D_T$,
\begin{eqnarray}
 \label{435}
 \displaystyle v(x,t):=\int_{-\infty}^{0}h(\xi)G(x,t;\xi,0)d\xi+\int_{0}^{t}G(x,t;X_\tau,\tau)\varphi(\tau)d\tau,
 \end{eqnarray}
where $\displaystyle \varphi\in C_{(\gamma)}^{0}\left((0,T]\right)$ satisfies
\begin{eqnarray}
   \label{11}   
      0=\int_{-\infty}^{0}h(\xi)G(X_t,t;\xi,0)d\xi+\int_{0}^{t}G(X_t,t;X_\tau,\tau)\varphi(\tau)d\tau.
 \end{eqnarray}
If we apply the Abel inverse operator $A$ defined by
$$\displaystyle (AF)(t)=\frac{1}{\pi}\frac{d}{dt}\int_{0}^{t}\frac{F(\eta)}{(t-\eta)^{\frac{1}{2}}}d\eta$$
on both sides of (\ref{11}), then from Chapter 14 in $\cite{Cannon}$ we have a equivalent Volterra integral equation of the second kind:
 \begin{eqnarray}
 \label{434}
   \varphi(t)=\psi(t) + \int_{0}^{t}H(t,\tau)\varphi(\tau)d\tau,
 \end{eqnarray}
where $\displaystyle \psi\in C_{(\gamma)}^{0}\left((0,T]\right)$ and $|H(t,\tau)|\leq C(t-\tau)^{2\gamma-2}$. The existence of $\displaystyle \varphi\in C_{(\gamma)}^{0}\left((0,T]\right)$ which satisfies $(\ref{434})$ can be proved similarly to the proof of $\textbf{Proposition~\ref{100}}$ that we will show later. Then $v$ is well-defined and solves $(\ref{4242})$ so that $v=u$ since $u$ is the unique solution of $(\ref{4242})$. By the jump relation, we have 
\begin{eqnarray}
 \displaystyle u_x(X_t^{-},t)=\int_{-\infty}^{0}h(\xi)G_x(X_t,t;\xi,0)d\xi+\varphi(t)+\int_{0}^{t}G_x(X_t,t;X_\tau,\tau)\varphi(\tau)d\tau
 \end{eqnarray}
so that $\displaystyle u_x(X_t^{-},t)\in C_{(\gamma)}^{0}((0,T])$. Since $T$ is arbitrary, we deduce $\displaystyle u_x(X_t^{-},t)$ is continuous on $(0,\infty)$.\\
To show $(\ref{888})$, let us fix $(x,t)\in D_T$ and let us define $\displaystyle D_{\epsilon,r}^{(t)}:=\{(\xi,\tau): r<\xi<X_\tau-\epsilon,\ \epsilon<\tau<t-\epsilon\}$ for each $\epsilon>0$ and $r\in\mathbb{R}$. By the Green's identity, we have
\begin{eqnarray}
 \displaystyle \frac{1}{2}(u_\xi G-uG_\xi)_\xi-(uG)_{\tau}=0\ \Longrightarrow\  \ \oint_{\partial D_{\epsilon,r}^{(t)}} \frac{1}{2}(u_\xi G-uG_\xi)d\tau+(uG)d\xi=0.
\end{eqnarray}
It  can be also shown that $\displaystyle \lim_{x\rightarrow -\infty}\sup_{0< t< T}|u_x(x,t)|=0$ by the properties of the Gaussian kernel. Hence we obtain another representation of $u$ by letting $\epsilon\rightarrow 0$, $r\rightarrow -\infty$,
\begin{eqnarray}
\label{777}
\displaystyle u(x,t)=\int_{-\infty}^{0}h(\xi)G(x,t;\xi,0)d\xi+\frac{1}{2}\int_{0}^{t}G(x,t;X_\tau,\tau)u_x(X_\tau^{-},\tau)d\tau.
\end{eqnarray}
Differentiating both sides of $(\ref{777})$ with respect to x  and applying the jump relation, we get
\begin{eqnarray}
\displaystyle \frac{1}{2}u_x(X_t^{-},t)=\int_{-\infty}^{0}h(\xi)G_x(X_t,t;\xi,0)d\xi+\frac{1}{2}\int_{0}^{t}G_x(X_t,t;X_\tau,\tau)u_x(X_\tau^{-},\tau)d\tau.
\end{eqnarray}
which implies $(\ref{888})$.
\end{proof}

\vskip1cm
\section{Proof of \textbf{of Theorem~\ref{mainthm}}} 
\label{Fsec.2}
To prove item $\ref{1.1}$, let $X$ be a continuous curve defined on $[0,\infty)$ and let $r_0<X_0$. Using the strong Markov property, we have, for $\displaystyle f\in C_c^{\infty}(\mathbb{R})$ with $\displaystyle\textrm{supp}\hspace{0.5mm}f\Subset (-\infty,X_t)$ and $0\leq s\leq t$,
\begin{eqnarray}
\displaystyle E_{r_0}[f(B_t)|\tau_{r_0}^{X}= s]=E_{X_{s},s}[f(B_{t})].
\end{eqnarray}
Thus we get
\begin{eqnarray}
\displaystyle E_{r_0}[f(B_t);\tau_{r_0}^{X}\leq t]=\int_{[0,t]}F_{r_0}^{X}(ds)E_{X_{s},s}[f(B_{t})].
\end{eqnarray}
It follows that item $\ref{1.1}$ 
 holds with $\displaystyle G_{0,t}^{X}(r_0,x)$ given by $(\ref{12345})$.\\
To prove item $\ref{1.2}$, from now on, $X$ is a H\"{o}lder continuous curve defined on $[0,\infty)$ with exponent $\gamma\in(1/2,1]$ and $X_0=0$. Comparing the definition \eqref{2345} of $u$ and $(\ref{777})$, using \eqref{12345}, we see the following equality:
\begin{eqnarray}
\displaystyle \int_{[0,t)}G_{\tau,t}(X_\tau,x)\int_{-\infty}^{0}h(\xi)F_{\xi}^{X}(d\tau)d\xi=-\frac{1}{2}\int_{0}^{t}G(x,t;X_\tau,\tau)u_x(X_\tau^{-},\tau)d\tau.
\end{eqnarray}
Set $\displaystyle F_{h}^{X}(d\tau):=\int_{-\infty}^{0} h(\xi) F_{\xi}^{X}(d\tau)d\xi$. 
\begin{prop}
\label{new}
$\displaystyle F_{h}^{X}(d\tau)=-\frac{1}{2}u_x(X_\tau^{-},\tau)d\tau$.
\end{prop}
For the proof of Proposition \ref{new}, we introduce the mass lost $\displaystyle \Delta_{I}^{X}(u)$, $I=[t_1,t_2]\subset [0,T]$, $t_1\leq t_2$, is defined by
\begin{eqnarray}
\displaystyle \Delta_{I}^{X}(u)=\int_{-\infty}^{X_{t_1}}u(r,t_1)dr-\int_{-\infty}^{X_{t_2}}u(r,t_2)dr.
\end{eqnarray}
If we see the right hand side of $(\ref{777})$, we can extend $u$ to $\bar{u}$ defined in $\{(x,t): x\in\mathbb{R},\ 0<t\leq T\}$ as 
\begin{eqnarray}
\displaystyle \bar{u}(x,t)=\int_{-\infty}^{0}h(\xi)G(x,t;\xi,0)d\xi+\frac{1}{2}\int_{0}^{t}G(x,t;X_\tau,\tau)u_x(X_\tau^{-},\tau)d\tau.
\end{eqnarray}
Then this satisfies the heat equation with $\displaystyle \lim_{(x,t)\rightarrow (y,0)}\bar{u}(x,t)=0$ for all $y\geq 0$ and also satisfies $\bar{u}(X_t,t)=0$ for all $0<t\leq T$. Moreover, by the properties of Gaussian kernel, we have 
\begin{eqnarray}
\displaystyle \lim_{x\rightarrow \infty}\sup_{0< t< T}|\bar{u}(x,t)|=0.
\end{eqnarray}
It follows that $\bar{u}(x,t)=0$ in $\{(x,t): x\geq X_t, 0< t\leq T\}$ by the weak maximum(minimum) principle. Thus we assume that $u$ is defined $\{(x,t): x\in\mathbb{R},\ 0\leq t\leq T\}$ such that it is $0$ in $\{(x,t): x\geq X_t,\ 0\leq t\leq T\}$.
\\ 
Heuristically
\begin{eqnarray*}
\displaystyle \Delta_{I}^{X}(u)= -\int \int_{t_1}^{t_2} u_t(x,t)dtdx= -\int \int_{t_1}^{t_2}\frac{1}{2}u_{xx}(x,t)dxdt= -\frac{1}{2}\int_{t_1}^{t_2} u_x(X_t^{-},t)dt.
\end{eqnarray*}
Since we do not control $u_{xx}$ at the moving boundary, we cannot make this argument rigorously. Thus we use a different approach.\\

\begin{proof}\hspace{-1mm}
\textbf{of Proposition~\ref{new}}
\vskip 0.1cm
It suffices to show
$$\displaystyle -\frac{1}{2}\int_{I}u_x(X_t^{-},t)dt=\Delta_{I}^{X}(u)=F_h^{X}(I).$$
If we integrate both sides of $(\ref{777})$, then
\begin{eqnarray*}
\displaystyle \int_{-\infty}^{\infty}u(x,t)dx= \int_{-\infty}^{\infty}\int_{-\infty}^{0}h(\xi)G(x,t;\xi,0)d\xi dx+ \int_{-\infty}^{\infty}\frac{1}{2}\int_{0}^{t}G(x,t;X_\tau,\tau)u_x(X_\tau^{-},\tau)d\tau dx.
\end{eqnarray*}
Applying Fubini's theorem, we get
\begin{eqnarray}
\displaystyle   \int_{-\infty}^{X_t}u(x,t)dx=\int_{-\infty}^{\infty}u(x,t)dx= \int_{-\infty}^{0}h(\xi)d\xi+\frac{1}{2}\int_{0}^{t}u_x(X_\tau^{-},\tau)d\tau .
\end{eqnarray}
Thus we get the first equality of the proposition,
\begin{eqnarray}
\displaystyle \Delta_{[t_1,t_2]}^{X}(u)=-\frac{1}{2}\int_{t_1}^{t_2}u_x(X_\tau^{-},\tau)d\tau.
\end{eqnarray}

From $(\ref{1.0})$ and item $\ref{1.1}$ of $\textbf{Theorem~\ref{mainthm}}$, we get
\begin{eqnarray}
\displaystyle P_{\xi}[\tau_\xi^{X} > t]=\int_{-\infty}^{X_t}G_{0,t}^{X}(\xi,x)dx.
\end{eqnarray}
For $0=t_1<t_2$, using Fubini's theorem again, we get
\begin{eqnarray*}
\displaystyle \Delta_{I}^{X}(u)&=&\int_{-\infty}^{0}h(\xi)d\xi-\int_{-\infty}^{X_{t_2}}\int_{-\infty}^{0} h(\xi)G_{0,t_2}^{X}(\xi,x)d\xi dx\\
&=&\int_{-\infty}^{0}h(\xi)d\xi-\int_{-\infty}^{0}h(\xi)P_{\xi}[\tau_\xi^{X} > t_{2}]d\xi=\int_{-\infty}^{0}h(\xi)P_{\xi}[0\leq\tau_\xi^{X} \leq t_{2}]d\xi.\\
\end{eqnarray*}
For $0<t_1<t_2$, similarly,
\begin{eqnarray*}
\displaystyle \Delta_{I}^{X}(u)&=&\int_{-\infty}^{X_{t_1}}\int_{-\infty}^{0} h(\xi)G_{0,t_1}^{X}(\xi,x)d\xi dx-\int_{-\infty}^{X_{t_2}}\int_{-\infty}^{0} h(\xi)G_{0,t_2}^{X}(\xi,x)d\xi dx\\
&=&\int_{-\infty}^{0}h(\xi)P_{\xi}[t_1< \tau_\xi^{X} \leq t_{2}]d\xi.
\end{eqnarray*}
Then for $I_\epsilon=[t_1-\epsilon, t_1]$, we get
\begin{eqnarray*}
\displaystyle \lim_{\epsilon\rightarrow 0}\Delta_{I_\epsilon}^{X}(u)&=&\lim_{\epsilon\rightarrow 0}-\frac{1}{2}\int_{I_\epsilon}u_x(X_t^{-},t)dt=0=\lim_{\epsilon\rightarrow 0}\int_{-\infty}^{0}h(\xi)P_{\xi}[t_1-\epsilon< \tau_\xi^{X} \leq t_1]d\xi\\
&=&\int_{-\infty}^{0}h(\xi)P_{\xi}[\tau_\xi^{X} = t_{1}]d\xi.
\end{eqnarray*}
Finally we conclude that
\begin{eqnarray}
\displaystyle \Delta_{I}^{X}(u)=\int_{-\infty}^{0}h(\xi)P_{\xi}[t_1\leq \tau_\xi^{X}\leq t_2]d\xi=F_h^{X}(I).
\end{eqnarray}
\end{proof}



By approximating the initial delta measure of $\textbf{Theorem~\ref{mainthm}}$, we prove the proposition below.

\begin{prop}
\label{890}
Let $X$ be a H\"{o}lder continuous curve on any finite interval in $[0,\infty)$ with exponent $\gamma\in(1/2,1]$ and let $X_0=0$. We fix $r_0<X_0=0$ and choose a sequence $\displaystyle\{h_n\}\subset C_c^{\infty}(\mathbb{R};\mathbb{R}_{+})$ with $\textrm{supp}\hspace{0.1cm}{h_n}=[r_0-\frac{1}{n},r_0+\frac{1}{n}]\Subset (-\infty,0)$ and $\displaystyle \lVert h_n \rVert_{1}=1$. For each $h_n$, there exists a corresponding solution $\displaystyle u_{n}$ with $\displaystyle -\frac{1}{2}\frac{\partial u_n}{\partial x}\Big|_{(X_{t}^{-},t)}=:p_{n}(t)$ in the sense of $\textbf{Theorem~\ref{thm1}}$. Then we have the following statements:
\begin{enumerate}
\label{100} 
\item There is a unique $p\in C([0,\infty))$ with $\displaystyle p(0)=0$ such that
\begin{eqnarray}
\label{111}
\displaystyle p(t)=-G_x(X_t,t;r_0,0)+\int_0^{t}G_x(X_t,t;X_\tau,\tau)p(\tau)d\tau\ for\ all\ t>0.
\end{eqnarray}
\label{222}
\item $p_n$ converges to $p$ in $\displaystyle C_{(\eta)}^{0}((0,T])$ for all $0<\eta<1/2$.
\end{enumerate}
\end{prop}
\vskip 0.3cm

We use the following lemma extensively to prove $\textbf{Proposition~\ref{890}}$:
\begin{lem}
\label{lem1}
$\displaystyle \int_0^{t}\tau^{\alpha_1}(t-\tau)^{\alpha_2}d\tau=\frac{\Gamma(1+\alpha_1)\Gamma(1+\alpha_2)}{\Gamma(2+\alpha_1+\alpha_2)}t^{1+\alpha_1+\alpha_2}$ for $\alpha_1, \alpha_2>-1$, where $\Gamma$ is the gamma function.
\end{lem}

\vskip 0.5cm
\begin{proof}\hspace{-1mm}
\textbf{of Proposition~\ref{890}}\\
Let $T_s>0$. Since $\displaystyle |G_x(X_t,t;X_\tau,\tau)|=\left|\frac{1}{\sqrt{2\pi (t-\tau)}}\left\{-\frac{X_{t}-X_{\tau}}{t-\tau}\right\}\exp{\left(-\frac{(X_t-X_\tau)^{2}}{2(t-\tau)}\right)}\right|\leq\frac{C_0}{(t-\tau)^{\frac{3}{2}-\gamma }}$, we deduce that for $\displaystyle q_1, q_2\in C([0,T_s])$, $0<t\leq T_s$,
\begin{eqnarray*}
\displaystyle \left|\int_0^{t}G_x(X_t,t;X_\tau,\tau)(q_1(\tau)-q_2(\tau))d\tau\right|\leq \int_0^{t}\frac{C_0\lVert q_1-q_2\rVert_{T_s}}{(t-\tau)^{\frac{3}{2}-\gamma}}d\tau\\
=C_1t^{\gamma-\frac{1}{2}}\lVert q_1-q_2\rVert_{T_s}\leq C_1T_s^{\gamma-\frac{1}{2}}\lVert q_1-q_2\rVert_{T_s}.
\end{eqnarray*}
We define $\displaystyle F:C([0,T_s])\rightarrow C([0,T_s])$ as, for $q\in C([0,T_s])$,
\begin{eqnarray}
\displaystyle (Fq)(t)=-G_x(X_t,t;r_0,0)+\int_0^{t}G_x(X_t,t;X_\tau,\tau)q(\tau)d\tau\ for\ t>0,
\end{eqnarray}
and  $(Fq)(0)=0$. Since $\displaystyle\lim_{t\rightarrow 0}G_x(X_t,t;r_0,0)=0$, it is well-defined. If we choose $T_s$ such that $\displaystyle C_1T_s^{\gamma-\frac{1}{2}}<1$, then $F$ is a contraction mapping so that $F$ has a unique fixed point since $C([0,T_s])$ is a Banach space. Let's call this $\displaystyle p_{T_s}$.\\
Now we assume that we have $p_{T_s}$ for some $T_s>0$. For $\displaystyle T^{\star}>T_s$, we define $\displaystyle H:C([T_s,T^{\star}])\rightarrow C([T_s,T^{\star}])$ as, for $q\in C([T_s,T^{\star}])$,
\begin{eqnarray*}
\displaystyle (Hq)(t)=-G_x(X_t,t;r_0,0)+\int_0^{T_s}G_x(X_t,t;X_\tau,\tau)p_{T_s}(\tau)d\tau+  \int_{T_s}^{t}G_x(X_t,t;X_\tau,\tau)q(\tau)d\tau.
\end{eqnarray*}
Then for $\displaystyle q_1, q_2\in C([T_s,T^{\star}])$, we have
\begin{eqnarray}
\displaystyle \lVert Hq_1-Hq_2\rVert_{\infty}\leq C_2(t-T_s)^{\gamma-\frac{1}{2}}\lVert q_1-q_2\rVert_{\infty}\leq C_2(T^{\star}-T_s)^{\gamma-\frac{1}{2}}\lVert q_1-q_2\rVert_{\infty}.
\end{eqnarray}
Similarly, if we choose $\displaystyle T^{\star}$ such that $\displaystyle  C_2(T^{\star}-T_s)^{\gamma-\frac{1}{2}}<1$, then $H$ is a contraction mapping so that $H$ has a unique fixed point since $\displaystyle C([T_s,T^{\star}])$ is a Banach space.\\
Therefore, if we have $p$ defined $[0,T_s]$, $p_{T_s}$,  then we can extend this to time $T_s+C_3$ where $C_3$ is an independent constant. Thus if we repeat this step inductively, we have $p$ defined on $[0,\infty)$ which satisfies $(\ref{111})$.
\\
We now prove that $p_n$ converges to $p$ in $\displaystyle C_{(\eta)}^{0}((0,T_s])$ for all sufficiently small $T_s>0$. By $(\ref{888})$, 
\begin{eqnarray}
\label{999}
\displaystyle p_n(t)=-\int_{-\infty}^{0}h_n(\xi)G_x(X_t,t;\xi,0)d\xi+\int_{0}^{t}G_x(X_t,t;X_\tau,\tau)p_n(\tau)d\tau.
\end{eqnarray}
For $0<t\leq T_s$, taking the difference between $(\ref{111})$ and $(\ref{999})$, we get
\begin{eqnarray*}
t^{1-\eta}|p_n(t)-p(t)|&\leq& t^{1-\eta}\left|\int_{-\infty}^{0}h_n(\xi)[G_x(X_t,t;\xi,0)-G_x(X_t,t;r_0,0)]d\xi\right|\\
&+&t^{1-\eta}\left|\int_{0}^{t}G_x(X_t,t;X_\tau,\tau)(p_n(\tau)-p(\tau))d\tau\right|.
\end{eqnarray*}
For the second term of the right hand side, we have
\begin{eqnarray*}
t^{1-\eta}\left|\int_{0}^{t}G_x(X_t,t;X_\tau,\tau)(p_n(\tau)-p(\tau))d\tau\right|\leq C_0t^{1-\eta}\int_{0}^{t}\frac{\lVert p_n-p\rVert_{T_s}^{(\eta)}}{(t-\tau)^{\frac{3}{2}-\gamma}\tau^{1-\eta}}d\tau\\
=C_1t^{\gamma-\frac{1}{2}}\lVert p_n-p\rVert_{T_s}^{(\eta)}\leq C_1T_s^{\gamma-\frac{1}{2}}\lVert p_n-p\rVert_{T_s}^{(\eta)}.
\end{eqnarray*}
Let us choose $T_s>0$ such that $C_1T_s^{\gamma-\frac{1}{2}}<1$. Then
\begin{eqnarray*}
\displaystyle (1-C_1T_s^{\gamma-\frac{1}{2}})\lVert p_n-p\rVert_{T_s}^{(\eta)}\leq t^{1-\eta}\left|\int_{-\infty}^{0}h_n(\xi)[G_x(X_t,t;\xi,0)-G_x(X_t,t;r_0,0)]d\xi\right|.\\
\displaystyle\leq\int_{-\infty}^{0}h_n(\xi)\sup_{\substack{0<t\leq T_s\\ \xi\in[r_0-\frac{1}{n},r_0+\frac{1}{n}]}}t^{1-\eta}\left|G_x(X_t,t;\xi,0)-G_x(X_t,t;r_0,0)\right|d\xi.
\end{eqnarray*}

For all sufficiently large $n$, $0<t\leq\frac{1}{n}$ and $\displaystyle r_0-\frac{1}{n}\leq\xi\leq r_0+\frac{1}{n}$, there exists $C_2$ such that 
\begin{eqnarray*}
\left|G_x(X_t,t;\xi,0)\right|= \left|\frac{1}{\sqrt{2\pi}}\left\{-\frac{X_t-\xi}{t^{\frac{3}{2}}}\right\}\exp{\left(-\frac{(X_t-\xi)^{2}}{2t}\right)}\right|\leq C_{2},
\end{eqnarray*}
Thus
\begin{eqnarray*}
\sup_{\substack{0<t\leq \frac{1}{n}\\ \xi\in[r_0-\frac{1}{n},r_0+\frac{1}{n}]}}t^{1-\eta}\left|G_x(X_t,t;\xi,0)-G_x(X_t,t;r_0,0)\right|\leq \sup_{0<t\leq \frac{1}{n}}2t^{1-\eta}C_{2}\leq \frac{2C_2}{n^{1-\eta}}.
\end{eqnarray*}

For all sufficiently large $n$, $\frac{1}{n}<t\leq T_s$ and $\displaystyle r_0-\frac{1}{n}\leq\xi\leq r_0+\frac{1}{n}$, since $\displaystyle|G_{xx}(\cdot,t;\cdot, 0)|\leq\frac{C_3}{t^{\frac{3}{2}}}$, we have
\begin{eqnarray*}
\left|G_x(X_t,t;\xi,0)-G_x(X_t,t;r_0,0)\right| \leq \frac{C_3|\xi-r_0|}{t^{\frac{3}{2}}}\leq\frac{C_3}{nt^{\frac{3}{2}}},
\end{eqnarray*}
Thus
\begin{eqnarray*}
\sup_{\substack{\frac{1}{n}<t\leq T_s\\ \xi\in[r_0-\frac{1}{n},r_0+\frac{1}{n}]}}t^{1-\eta}\left|G_x(X_t,t;\xi,0)-G_x(X_t,t;r_0,0)\right|\leq \sup_{\frac{1}{n}<t\leq T_s}\frac{C_3}{nt^{\frac{1}{2}+\eta}}\leq\frac{C_3}{n^{\frac{1}{2}-\eta}}.
\end{eqnarray*}
Therefore, we conclude that $p_n$ converges to $p$ in $\displaystyle C_{(\eta)}^{0}((0,T_s])$ for all sufficiently small $T_s>0$.\\
To extend from $T_s$ to $T^{\star}$, assuming that  $p_n$ converges to $p$ in $\displaystyle C_{(\eta)}^{0}((0,T_s])$ for some $T_s>0$ and writing $\displaystyle \lVert p_n-p\rVert_{[T_s,T^{\star}]}=\displaystyle\sup_{\tau\in[T_s,T^{\star}]} \left|p_n(\tau)-p(\tau)\right|$, for $\displaystyle T_s\leq t\leq T^{\star}$, we deduce that
\begin{eqnarray*}
|p_n(t)-p(t)|&\leq& \left|\int_{-\infty}^{0}h_n(\xi)[G_x(X_t,t;\xi,0)-G_x(X_t,t;r_0,0)]d\xi\right|\\
&+&\left|\int_{0}^{T_s}G_x(X_t,t;X_\tau,\tau)(p_n(\tau)-p(\tau))d\tau\right|+\left|\int_{T_s}^{t}G_x(X_t,t;X_\tau,\tau)(p_n(\tau)-p(\tau))d\tau\right|\\
&\leq& C_3\int_{-\infty}^{0}h_n(\xi)\frac{|\xi-r_0|}{t^{\frac{3}{2}}}d\xi+C_0\int_{0}^{T_s}\frac{\lVert p_n-p\rVert_{T_s}^{(\eta)}}{(t-\tau)^{\frac{3}{2}-\gamma}\tau^{1-\eta}}d\tau+C_0\int_{T_s}^{t}\frac{\lVert p_n-p\rVert_{[T_s,T^{\star}]}}{(t-\tau)^{\frac{3}{2}-\gamma}}d\tau\\
&\leq& \frac{C_3}{nt^{\frac{3}{2}}}+C_0\int_{0}^{T_s}\frac{\lVert p_n-p\rVert_{T_s}^{(\eta)}}{(T_s-\tau)^{\frac{3}{2}-\gamma}\tau^{1-\eta}}d\tau+C_4\lVert p_n-p\rVert_{[T_s,T^{\star}]}(t-T_s)^{\gamma-\frac{1}{2}}\\
&\leq&\frac{C_3}{nT_s^{\frac{3}{2}}} +C_1T_s^{\gamma-\frac{3}{2}+\eta}\lVert p_n-p\rVert_{T_s}^{(\eta)}+C_4\lVert p_n-p\rVert_{[T_s,T^{\star}]}(T^{\star}-T_s)^{\gamma-\frac{1}{2}}.
\end{eqnarray*}
Let us choose $T^{\star}>T_s$ such that $C_4(T^{\star}-T_s)^{\gamma-\frac{1}{2}}<1$, then we have
\begin{eqnarray*}
(1-C_4(T^{\star}-T_s)^{\gamma-\frac{1}{2}})\lVert p_n-p\rVert_{[T_s,T^{\star}]}\leq \frac{C_3}{nT_s^{\frac{3}{2}}}+C_1T_s^{\gamma-\frac{3}{2}+\eta}\lVert p_n-p\rVert_{T_s}^{(\eta)}.
\end{eqnarray*}
The right term vanishes when $n$ goes to $\infty$ so that $p_n$ converges to $p$ in $\displaystyle C_{(\eta)}^{0}((0,T_s+C_5])$ for some independent constant $C_5>0$. By repeating this argument inductively, it follows that $p_n$ converges to $p$ in $\displaystyle C_{(\eta)}^{0}((0,T])$.
\end{proof}

\vskip 0.5cm
Now we can prove $p$ from Proposition \ref{890} is the density function of $\displaystyle F_{r_0}^{X}$.\\
Hence we have
\begin{eqnarray*}
\displaystyle \lim_{n\rightarrow \infty}\int_{-\infty}^{0}h_{n}(\xi)P_{\xi}(\tau_\xi^{X}\in I)d\xi=\lim_{n\rightarrow \infty}\int_{I}p_n(t)dt=\int_{I}p(t)dt.
\end{eqnarray*}
For $I=[0,t]\subset[0,T]$, $\displaystyle P_{\xi}(\tau_\xi^{X}\in I)$ is an increasing function of $\xi$, so if we choose $h_n$ so that $\displaystyle\lim_{n\rightarrow\infty}\int_{r_0}^{r_0+\frac{1}{n}}h_n(\xi)d\xi=1$, we get
\begin{eqnarray}
\displaystyle \lim_{\xi\rightarrow r_0^{+}}P_{\xi}(\tau_\xi^{X}\in I)\leq \lim_{n\rightarrow \infty}\int_{r_0}^{r_0+\frac{1}{n}}h_{n}(\xi)P_{\xi}(\tau_\xi^{X}\in I)d\xi=\lim_{n\rightarrow \infty}\int_{-\infty}^{0}h_{n}(\xi)P_{\xi}(\tau_\xi^{X}\in I)d\xi.
\end{eqnarray}
Similarly, if we choose $h_n$ so that $\displaystyle\lim_{n\rightarrow\infty}\int_{r_0-\frac{1}{n}}^{r_0}h_n(\xi)d\xi=1$, we get
\begin{eqnarray}
\displaystyle  \lim_{n\rightarrow \infty}\int_{-\infty}^{0}h_{n}(\xi)P_{\xi}(\tau_\xi^{X}\in I)d\xi=\lim_{n\rightarrow \infty}\int_{r_0-\frac{1}{n}}^{r_0}h_{n}(\xi)P_{\xi}(\tau_\xi^{X}\in I)d\xi\leq \lim_{\xi\rightarrow r_0^{-}}P_{\xi}(\tau_\xi^{X}\in I).
\end{eqnarray}
Therefore, we obtain that
\begin{eqnarray}
\displaystyle F_{r_0}^{X}([0,t])=P_{r_0}(\tau_{r_0}^{X}\leq t)=\int_{0}^{t}p(\tau)d\tau,
\end{eqnarray}
which implies that $p$ is the density function of $\displaystyle F_{r_0}^{X}$, thus item $\ref{1.2}$ is proved.\\ By $\textbf{Theorem~\ref{thm1}}$ and the properties of the Gaussian kernel, $\displaystyle G_{0,t}^{X}(r_0,x)$ solves $(\ref{1})$, $(\ref{2})$ and $(\ref{3})$. Hence $G^{X}$ is the Green function of the heat equation with Dirichlet boundary condition which implies item $\ref{1.4}$. Furthermore, $\displaystyle G_{0,t}^{X}(r_0,x)$ can be written as
\begin{eqnarray}
\displaystyle G_{0,t}^{X}(r_0,x)=G_{0,t}(r_0,x)-\int_{0}^{t}G_{\tau,t}(X_\tau,x)p(\tau)d\tau.
\end{eqnarray}
Differentiating with respect to x, applying the jump relation and $(\ref{111})$, we have
\begin{eqnarray}
\displaystyle \frac{\partial}{\partial x}G_{0,t}^{X}(r_0,x)\Big|_{x=X_t^{-}}&=&\frac{\partial}{\partial x}G_{0,t}(r_0,X_t)-p(t)-\int_{0}^{t} \frac{\partial}{\partial x}G_{\tau,t}(X_\tau,X_t)p(\tau)d\tau\\
&=&-2p(t).
\end{eqnarray}
Thus
$\displaystyle p(t)=-\frac{1}{2}\frac{\partial}{\partial x}G_{0,t}^{X}(r_0,x)\Big|_{x=X_t^{-}}$, so it proves item $\ref{1.3}$ of $\textbf{Theorem~\ref{mainthm}}$.

\section*{Acknowledgement}
I am thankful to Errico Presutti, Anna De Masi for helpful suggestions and Victor Perez Abreu for useful comments.


\vskip 1cm


\begin{thebibliography}{}


\bibitem{RSS} L. M. Ricciardi; L. Sacerdote; S. Sato: \textit{On an Integral Equation for First-Passage-Time Probability Densities}, Journal of Applied Probability,  \textbf{Vol. 21, No. 2.} 302--314 (1984).

\bibitem{Cannon} J.R. Cannon: \textit{The One-Dimensional Heat Equation}, Addison-Wesley Publishing Company (1984).

\bibitem{KS} Ioannis Karatzas,  Steve E.Shreve: \textit{Brownian Motion and Stochastic Calculus}, Springer (1991).

\bibitem{PS} G. Peskir, A. Shiryaev: \textit{Optimal Stopping and Free-Boundary Problems}, Birkh\"{a}user  (2006).

\bibitem{Fasano} A. Fasano, \textit{Mathematical models of some diffusive processes with free boundaries}, SIMAI e-Lecture Notes (2008).

\bibitem{BD} K. Borovkov, A.N. Downes: \textit{On boundary crossing probabilities for diffusion processes}, Stochastic Processes and their Applications, Volume 120, Issue 2, February 2010, Pages
105--129.

\bibitem{TM} T. Taillefumier, M. Magnasco:  \textit{A Transition to Sharp Timing in Stochastic Leaky Integrate-and-Fire Neurons Driven by Frozen Noisy Input}, Neural Computation  \textbf{Vol. 26, No. 5.} 819--859 (2014).

\bibitem{AP} L. Alili, P. Patie: \textit{Boundary crossing identities for Brownian motion and some nonlinear ode's}, Proc. Amer. Math. Soc. \textbf{142} , 3811--3824 (2014).

\bibitem{CDGP} G. Carinci, A. De Masi, C. Giardin\`{a}, E. Presutti: \textit{Free Boundary Problems in PDEs and Particle Systems}, Springer (2016).




\end{thebibliography}
\end{document}